\documentclass{amsart}

\usepackage[utf8]{inputenc}

\usepackage{amsfonts}
\usepackage{amsmath}
\usepackage{amsthm}
\usepackage{amssymb}
\usepackage{xcolor}
\usepackage{fullpage}
\usepackage{hyperref}

\DeclareMathOperator{\cay}{Cay}

\usepackage{tikz}
\usepackage{tkz-graph}
\usetikzlibrary{arrows.meta,calc,decorations.markings,math,arrows.meta,quotes}

\begin{document}

\def\PB{\color{purple}}
\def\JT{\color{blue}}

\newcommand{\C}{\mathbb{C}}
\newcommand{\F}{\mathbb{F}}
\newcommand{\R}{\mathbb{R}}
\newcommand{\Q}{\mathbb{Q}}
\newcommand{\N}{\mathbb{N}}
\newcommand{\Z}{\mathbb{Z}}

\newtheorem{thm}{Theorem}
\newtheorem{prop}[thm]{Proposition}
\newtheorem{coro}[thm]{Corollary}
\newtheorem{lem}[thm]{Lemma}
\newtheorem{conj}[thm]{Conjecture}

\theoremstyle{definition}
\newtheorem{defn}[thm]{Definition}
\newtheorem{rem}[thm]{Remark}
\newtheorem{ex}[thm]{Example}
\newtheorem{exs}[thm]{Examples}
\newtheorem{obs}[thm]{Observation}

\numberwithin{thm}{section}

\title{Cops and robbers on directed and undirected abelian Cayley graphs}

\begin{abstract}  
We show that the cop number of directed and undirected Cayley graphs on abelian groups has an upper bound of the form of $O(\sqrt{n})$, where $n$ is the number of vertices, by introducing a refined inductive method. With our method, we improve the previous upper bound on cop number for undirected Cayley graphs on abelian groups, and we establish an upper bound on the cop number of directed Cayley graphs on abelian groups. We also use Cayley graphs on abelian groups to construct new \emph{Meyniel extremal families}, which contain graphs of every order $n$ with cop number $\Theta(\sqrt{n})$.
\end{abstract}

\thanks{The authors are partially supported by the Natural Sciences and Engineering Research Council of Canada (NSERC) and the Fonds de Recherche du Québec – Nature et technologies (FRQNT)}


\author{Peter Bradshaw}
\address{Department of Mathematics, Simon Fraser University, Vancouver, Canada}
\email{peter\_bradshaw@sfu.ca}

\author{Seyyed Aliasghar Hosseini}
\address{Department of Mathematics, Simon Fraser University, Vancouver, Canada}
\email{seyyed\_aliasghar\_hosseini@sfu.ca}

\author{J\'er\'emie Turcotte}
\address{D\'{e}partment de math\'{e}matiques et de statistique, Universit\'{e} de Montr\'{e}al, Montr\'eal, Canada}
\email{mail@jeremieturcotte.com}
\urladdr{www.jeremieturcotte.com}

\maketitle

\section{Introduction}
We study the game of cops and robbers, a game in which a team of cops attempts to capture a robber while playing on the vertices of a graph. The game is played on a graph $\Gamma$ which is finite and connected, and can be either undirected or directed. The cops play as a team against the robber. Before the game starts, each cop chooses a starting vertex on $\Gamma$. The robber then does the same. The game alternates between \emph{cop turns} and \emph{robber turns}, with the first turn being a cop turn. On a cop turn, each cop can move to a neighbouring vertex or may choose to pass. The robber then has the same options on a robber turn. There is no restriction preventing two or more cops from sharing the same vertex. If one of the cops ever shares a vertex with the robber, then we will say that the robber is \emph{captured}, and capturing the robber is the cops' objective in order to win the game. On the other hand, if the cops never manage to capture the robber, then we say that the robber wins. The game is played with full information. The cop number of the graph $\Gamma$, written $c(\Gamma)$, is the minimum number of cops needed for a strategy that ensures the cops' victory.

The game of cops and robbers was first introduced for undirected graphs in \cite{Quilliot} by Quilliot, as well as in \cite{Nowakowski} by Nowakowski and Winkler. The concept of cop number was introduced shortly afterwards by Aigner and Fromme in \cite{Aigner}. The cop number is well-studied on many classes of graphs; bounds are known, for example, for graphs of high girth \cite{FranklGirth}, Cayley graphs \cite{PB, FranklGirth, FranklCayley, HAMIDOUNE1987289}, intersection graphs \cite{Gavenciak}, and graphs with certain forbidden subgraphs \cite{Joret2008TheCA, Masood}.

The game of cops and robbers can also be adapted to directed graphs, or digraphs, by making certain modifications. First, we require the digraphs on which the game is played to be strongly connected. Second, when a cop or robber moves along an arc to an adjacent vertex, we require that the cop or robber move in the same direction as the arc. This is in contrast to undirected graphs, in which a cop or robber may move along an edge in any direction, as edges have no orientation. The game of cops and robbers was first considered on digraphs by Hamidoune in \cite{HAMIDOUNE1987289}, and this directed version of the game has gained popularity recently; see, for example, \cite{Frieze,oriented2, Seyyedthesis, Seyyed, Seamone}.

Perhaps the furthest reaching and most famous question regarding the cop number is Meyniel's conjecture, which asks whether the cop number of any connected graph on $n$ vertices can be bounded by $O(\sqrt{n})$. Frankl first mentions Meyniel's conjecture for undirected graphs in \cite{FranklGirth}, and Baird and Bonato ask whether Meyniel's conjecture holds for strongly connected digraphs in \cite{Baird}. Meyniel's conjecture is known, for example, to hold for undirected graphs of diameter 2 \cite{lupeng, wagner}. The first author has also shown in \cite{PB} that the cop number of Cayley graphs on abelian groups satisfies Meyniel's conjecture, with an upper bound of $7\sqrt{n}$. Of course, the cop number is bounded above by a constant for many graph classes, such as graphs of bounded genus \cite{Aigner, QUILLIOT198589, Schroder2001}, graphs of bounded treewidth \cite{Joret2008TheCA} and graphs without long induced paths \cite{Joret2008TheCA}. In this paper, we will generalize the methods of \cite{PB} and \cite{FranklCayley} to both improve the upper bound for the cop number of Cayley graphs on abelian groups and show that directed Cayley graphs on abelian groups also satisfy Meyniel's conjecture, which will make these graph classes among the few large classes known to satisfy the conjecture.

Our paper is divided into multiple sections. In Section \ref{lemmasection}, we prove a general lemma about the cop number of Cayley graphs and digraphs on abelian groups. In Section \ref{undirectedsection}, we show that the cop number of an undirected Cayley graph on an abelian group of $n$ elements can be bounded by about $0.94\sqrt{n}+\frac{7}{2}$ and show that some improvements are possible by considering the prime decomposition of $n$. In Section \ref{directedsection}, we use the same methods to bound the cop number of a directed Cayley digraph on an abelian group of $n$ elements by about $1.33\sqrt{n}+2$. 
In Section \ref{lowersection}, we construct, for an infinite number of values $n$, undirected Cayley graphs on abelian groups of $n$ elements with cop number $\frac{1}{2}\sqrt{n}$, and directed Cayley graphs on abelian groups of $n$ elements with cop number $\sqrt{n}$. With a simple modification of these constructions, we obtain families of graphs on $n$ vertices, for any integer $n \geq 1$, with cop number $\Theta(\sqrt{n})$, which gives new Meyniel extremal families of graphs and digraphs. To the authors' knowledge, the family of digraphs that is obtained has the largest cop number in terms of $n$ of any known digraph construction. Finally, in Section \ref{furtherdirections}, we discuss possible improvements and further directions.


\section{Notation and a general strategy}\label{lemmasection}

In this section, we will establish some notation and outline our general approach to capturing a robber on a Cayley digraph on an abelian group. All groups that we consider in this paper are abelian. A \emph{directed Cayley graph} on an abelian group is defined as follows:

\begin{defn}
Let $(G,+)$ be a finite abelian group, and let $S \subseteq G$ be a generating set of $G$ with $0_G\notin S$. The \emph{Cayley graph} $\Gamma$ generated by $G$ and $S$ is defined as follows:
\begin{itemize}
\item $V(\Gamma) = G$;
\item For any $u,v \in G$, $\Gamma$ contains the arc $(uv)$ if and only if $v-u \in S$.
\end{itemize}
We often write $\cay(G,S)$ to refer to the Cayley digraph generated by $G$ and $S$.
\end{defn}
We will often refer to directed Cayley graphs on abelian groups as \emph{directed abelian Cayley graphs} or \emph{abelian Cayley digraphs}. 
In this definition, the requirement that $S$ generate $G$ ensures that the digraph $\cay(G,S)$ is strongly connected. We recall that in the game of cops on robbers on directed graphs, cops and robbers must traverse edges according to their orientations, so a graph must be strongly connected in order to allow a cop or robber to reach any vertex from any other vertex. We note that for a Cayley graph on an abelian group $G$ generated by a set $S \subseteq G$, if $S = -S$, then all arcs of $\cay(G,S)$ are bidirectional. In this case, the game of cops and robbers on the directed graph $\cay(G,S)$ is equivalent to the game on the undirected graph obtained from $\cay(G,S)$ by replacing each arc with an undirected edge and removing parallel edges. Therefore, when we wish to consider the game of cops and robbers on an undirected Cayley graph on an abelian group, we will require that $S = -S$, and we will regard $\cay(G,S)$ as an undirected graph. We often refer to an undirected Cayley graph on an abelian group as an \emph{abelian Cayley graph}.

When playing cops and robbers on a Cayley digraph on an abelian group $G$ generated by $S \subseteq G$, we imagine that at each turn, a cop or robber occupies some group element $g \in G$. In the Cayley digraph $\cay(G,S)$, the vertex $g$ has an out-neighbor $g+s$ for each $s \in S$, and thus we imagine that our cop or robber has a list of possible moves corresponding to the elements of $S$. This cop or robber may choose any element $s \in S$ on its turn and move to the group element $g + s \in G$. We call this \textit{playing the move $s$}. When a cop or robber stays at its current vertex, we say that the cop or robber plays the move $0_G$. To capture the robber, we will let our cops follow a strategy that makes certain robber moves $s \in S$ unsafe for the robber. As we make certain robber moves unsafe, the robber's list of possible moves will become shorter, and the robber's movement options will become more limited. As the robber's movement becomes more limited, it will become easier for the cops to make even more robber moves unsafe, and we will be able to limit the robber's movement further. Eventually, we will make every move unsafe for the robber, and the robber will have no way to avoid capture.  The precise meaning of an unsafe move is discussed below.

The approach of capturing the robber by reducing the number of safe robber moves is introduced by Frankl in \cite{FranklCayley}, which is itself inspired by the methods used by Hamidoune in \cite{HAMIDOUNE1987289}. Frankl shows that on an undirected abelian Cayley graph, one cop can usually make two robber moves unsafe, so the number of cops required to capture a robber is about half the size of the graph's generating set, as shown in the following theorem.
\begin{thm} \cite{FranklCayley}\label{thmFrankl}
If $\Gamma$ is a Cayley graph on an abelian group with generating set $S$ such that $S=-S$ and $0_G\notin S$, then
$$c(\Gamma)\leq \left\lceil \frac{|S|+1}{2}\right\rceil.$$
\end{thm}

When considering cops and robbers on a Cayley digraph on an abelian group $G$ generated by $S \subseteq G$, we will often define another set $T \subseteq S$ consisting of all of the moves of $S$ that the robber can still play safely, and we will assume that the robber chooses a move from $T$ on each turn in order to avoid unsafe moves. We will call $T$ the robber's \emph{moveset}. In other words, we will often consider a restricted version of the game in which on every turn, the robber is forced to play a move from a set $T$.

The following definition, which originally appears in \cite{PB} in a slightly different form, is closely related to the concept of limiting the robber's moves. 

\begin{defn}\label{accountsdef}
Let $G$ be an abelian group, and let $S \subseteq G$ and $0_G\notin S$. Given an element $a \in S$, we say that an element $k\in G\setminus\{0_G\}$ accounts for $a$ (with respect to $S$) if there exists an element $b \in S\cup \{0_G\}$ such that $a - b=k$.
\end{defn}

We give some intuition behind the reason that this concept is useful in limiting the moves of a robber on an abelian Cayley digraph. Consider an abelian group $G$ generated by a set $S$, and suppose that a game of cops and robbers is played on $\cay(G,S)$ in which the robber has a moveset $T \subseteq S$. Suppose that some element $k \in G$ accounts for a robber move $a \in T$. Then there exists an element $b \in S\cup \{0_G\}$ such that $a - b = k$. If the robber occupies a vertex $r \in G$, then a cop $C$ at $r + k$ can prevent the robber from playing $a$; if the robber plays $a$, then $C$ can play $b$ to capture the robber. Furthermore, if the robber plays another move $a' \in T$, then $C$ can also play $a'$ and maintain a difference of $k$ with the robber. Thus, on each subsequent turn, the robber must not play $a$, and we see that a cop $C$ at $r + k$ has a strategy to essentially remove $a$ from the robber's moveset. Similarly, a cop $C$ at $r + \gamma k$ for some nonnegative integer $\gamma$ can also essentially remove $a$ from the robber's moveset by considering the following observation. If the robber plays the move $a$ $\gamma$ times, then $C$ can respond with $b$ each time, and $C$ will capture the robber. As $C$ can maintain its ``difference" in $G$ with the robber by copying each move $a' \neq a$ that the robber plays, the robber can only play $a$ a finite number of times before being captured, and hence the robber must eventually abandon the move $a$.  We illustrate this concept in Figure  \ref{figurestrategy}. This strategy of using a cop to ``copy" the robber's moves and eventually prevent the robber from playing a certain move previously appears in \cite{FranklCayley} and \cite{HAMIDOUNE1987289}.

\begin{figure}[h]
\begin{tikzpicture}[scale=1.2,
dot/.style = {circle, fill, minimum size=#1,
inner sep=0pt, outer sep=0pt},
dot/.default = 5pt
]

\node [red, dot, label=left:$r$] (0,0) at (0,0) {};
\node [dot, label=right:$r+k$] (2,0) at (2,0) {};
\node [dot, label=right:$r+2k$] (4,0) at (4,0) {};
\node [dot, label=right:$r+3k$] (6,0) at (6,0) {};
\node [dot, label=right:$r+4k$] (8,0) at (8,0) {};
\node [blue, dot, label=right:$r+5k$] (10,0) at (10,0) {};

\node [dot, label=above:$r+a$] (1,1) at (1,1) {};
\node [dot, label=above:$(r+a)+k$] (3,1) at (3,1) {};
\node [dot, label=above:$(r+a)+2k$] (5,1) at (5,1) {};
\node [dot, label=above:$(r+a)+3k$] (7,1) at (7,1) {};
\node [dot, label=above:$(r+a)+4k$] (9,1) at (9,1) {};

\draw[-{Latex[scale=1.5]}] (0,0) to["$a$"] (1,1);
\draw[-{Latex[scale=1.5]}] (2,0) to["$a$"] (3,1);
\draw[-{Latex[scale=1.5]}] (4,0) to["$a$"] (5,1);
\draw[-{Latex[scale=1.5]}] (6,0) to["$a$"] (7,1);
\draw[-{Latex[scale=1.5]}] (8,0) to["$a$"] (9,1);

\draw[-{Latex[scale=1.5]}] (2,0) to["$b$" {above right}] (1,1);
\draw[-{Latex[scale=1.5]}] (4,0) to["$b$" {above right}] (3,1);
\draw[-{Latex[scale=1.5]}] (6,0) to["$b$" {above right}] (5,1);
\draw[-{Latex[scale=1.5]}] (8,0) to["$b$" {above right}] (7,1);
\draw[-{Latex[scale=1.5]}] (10,0) to["$b$" {above right}] (9,1);

\node [dot, label=below:$r+a'$] (1,-4) at (1,-4) {};
\node [dot, label=below:$(r+a')+k$] (3,-4) at (3,-4) {};
\node [dot, label=below:$(r+a')+2k$] (5,-4) at (5,-4) {};
\node [dot, label=below:$(r+a')+3k$] (7,-4) at (7,-4) {};
\node [dot, label=below:$(r+a')+4k$] (9,-4) at (9,-4) {};
\node [dot, label=below:$(r+a')+5k$] (11,-4) at (11,-4) {};

\node [dot] (2,-3) at (2,-3) {};
\node [dot] (4,-3) at (4,-3) {};
\node [dot] (6,-3) at (6,-3) {};
\node [dot] (8,-3) at (8,-3) {};
\node [dot] (10,-3) at (10,-3) {};

\draw[-{Latex[scale=1.5]}] (1,-4) to["$a$"] (2,-3);
\draw[-{Latex[scale=1.5]}] (3,-4) to["$a$"] (4,-3);
\draw[-{Latex[scale=1.5]}] (5,-4) to["$a$"] (6,-3);
\draw[-{Latex[scale=1.5]}] (7,-4) to["$a$"] (8,-3);
\draw[-{Latex[scale=1.5]}] (9,-4) to["$a$"] (10,-3);

\draw[-{Latex[scale=1.5]}] (3,-4) to["$b$" {above right}] (2,-3);
\draw[-{Latex[scale=1.5]}] (5,-4) to["$b$" {above right}] (4,-3);
\draw[-{Latex[scale=1.5]}] (7,-4) to["$b$" {above right}] (6,-3);
\draw[-{Latex[scale=1.5]}] (9,-4) to["$b$" {above right}] (8,-3);
\draw[-{Latex[scale=1.5]}] (11,-4) to["$b$" {above right}] (10,-3);

\draw[-{Latex[scale=1.5]}] (0,0) to["$a'$"] (1,-4);
\draw[-{Latex[scale=1.5]}] (2,0) to["$a'$"] (3,-4);
\draw[-{Latex[scale=1.5]}] (4,0) to["$a'$"] (5,-4);
\draw[-{Latex[scale=1.5]}] (6,0) to["$a'$"] (7,-4);
\draw[-{Latex[scale=1.5]}] (8,0) to["$a'$"] (9,-4);
\draw[-{Latex[scale=1.5]}] (10,0) to["$a'$"] (11,-4);

\draw[-{Latex[scale=1.5]}] (1,1) to["$a'$"] (2,-3);
\draw[-{Latex[scale=1.5]}] (3,1) to["$a'$"] (4,-3);
\draw[-{Latex[scale=1.5]}] (5,1) to["$a'$"] (6,-3);
\draw[-{Latex[scale=1.5]}] (7,1) to["$a'$"] (8,-3);
\draw[-{Latex[scale=1.5]}] (9,1) to["$a'$"] (10,-3);

\end{tikzpicture}
\caption{The figure shows a cop guarding a robber move in an abelian Cayley graph. The robber's vertex is labelled $r$, and each arc is labelled with its corresponding generating element. The values $a$ and $b$ are generating elements, and $a - b = k$. Here, a cop occupies $r + 5k$, so the difference between the cop and robber's positions is $5k$. If the robber plays $a$, then the cop will play $b$, and the difference between the cop and robber will decrease to $4k$. If the robber continues to play $a$, then the cop may continue to play $b$, decreasing the difference between the cop and robber's positions to $3k$, then $2k$, then $k$, and finally $0$. If the robber plays a different move $a'$, then the cop can also play $a'$ and maintain its difference with the robber. 
} 
\label{figurestrategy}
\end{figure}
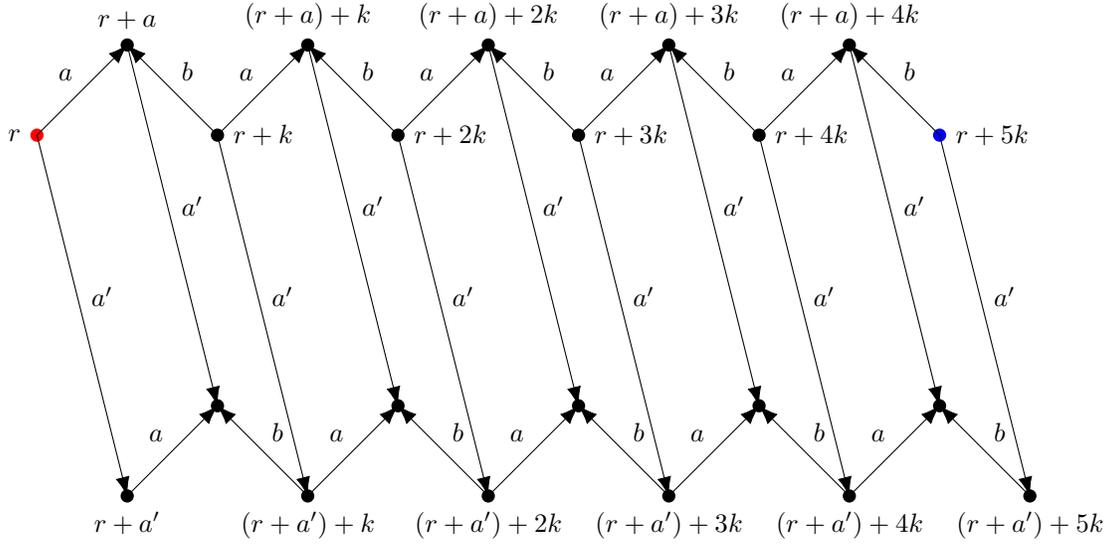

Thus, when will say that a robber move is unsafe, we will mean that a cop is ``guarding" this robber move as described above, and the robber can only play this move finitely many times before being captured by the guarding cop. As in \cite{PB}, we will use the fact that when $T$ is large, one element $k$ can account for many elements of $T$. Figure \ref{figuremultiguard} shows a local structure that appears in abelian Cayley digraphs with one group element accounting for many generating elements. The figure gives some intuition for how a single group element accounting for many generating elements allows a single cop to guard many robber moves.

\begin{figure}[h]
\center
\begin{tikzpicture}[scale=1.75,
dot/.style = {circle, fill, minimum size=#1,
inner sep=0pt, outer sep=0pt},
dot/.default = 5pt
]

\node [dot, label=below left:$r$] (-0.3,0) at (-1,0) {};
\node [dot, label=below right:$r+k$] (3,0) at (3,0) {};

\node [dot] (1,-1.9) at (1,-1.9) {};
\node [dot] (1,-1.1) at (1,-1.1) {};
\node [dot] (1,-0.3) at (1,-0.3) {};
\node [dot] (1,0.3) at (1,0.3) {};
\node [dot] (1,1.1) at (1,1.1) {};
\node [dot] (1,1.9) at (1,1.9) {};

\draw[-{Latex[scale=1.5]}] (-1,0) to["$a_1$" {below=4pt}] (1,-1.9);
\draw[-{Latex[scale=1.5]}] (-1,0) to["$a_2$" {below=1pt}] (1,-1.1);
\draw[-{Latex[scale=1.5]}] (-1,0) to["$a_3$" {below}] (1,-0.3);
\draw[-{Latex[scale=1.5]}] (-1,0) to["$a_4$" {above}] (1,0.3);
\draw[-{Latex[scale=1.5]}] (-1,0) to["$a_5$" {above=1pt}] (1,1.1);
\draw[-{Latex[scale=1.5]}] (-1,0) to["$a_6$" {above=4pt}] (1,1.9);

\draw[-{Latex[scale=1.5]}] (3,0) to["$b_1$" {below=4pt}] (1,-1.9);
\draw[-{Latex[scale=1.5]}] (3,0) to["$b_2$" {below=1pt}] (1,-1.1);
\draw[-{Latex[scale=1.5]}] (3,0) to["$b_3$" {below}] (1,-0.3);
\draw[-{Latex[scale=1.5]}] (3,0) to["$b_4$" {above}] (1,0.3);
\draw[-{Latex[scale=1.5]}] (3,0) to["$b_5$" {above=1pt}] (1,1.1);
\draw[-{Latex[scale=1.5]}] (3,0) to["$b_6$" {above=4pt}] (1,1.9);

\end{tikzpicture}
\caption{The figure shows a subgraph of an abelian Cayley digraph $\Gamma$. The generating set of $\Gamma$ contains six generating elements $a_1, \dots, a_6$ and six generating elements $b_1, \dots, b_6$, satisfying $a_1 - b_1 = \dots = a_6 - b_6$. Therefore, the group element $k = a_1 - b_1$ accounts for all six generating elements $a_1, \dots, a_6$, and thus if the robber occupies the vertex $r$, a cop at $r+k$ can guard all moves $a_1, \dots, a_6$.}
\label{figuremultiguard}
\end{figure}
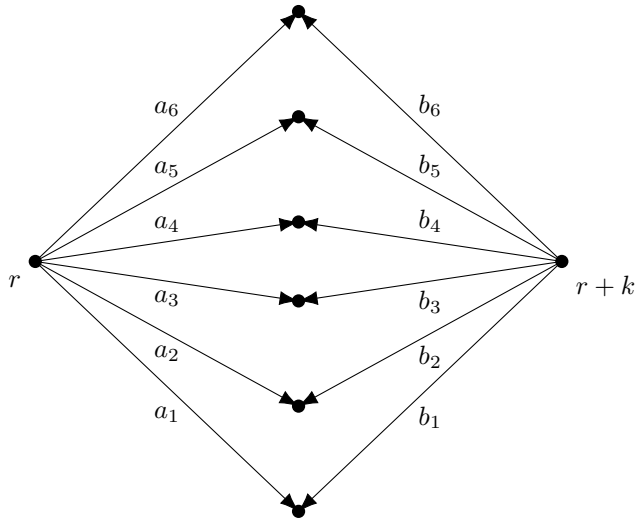

To avoid repeating the same conditions, we define the following notation.
\begin{defn}
    We define $$\mathcal G_d=\{(G,S,T) : G \text{ is a finite abelian group, } S \text{ is a generating set of $G$, }0_G\notin S,\textrm{ and }T\subseteq S\}$$
    and $$\mathcal G_u=\{(G,S,T)\in \mathcal G_d : S=-S\}.$$
     We also define $\mathcal D=\{(n,s,t)\in \mathbb N^3 : n\geq 1$ and $n-1\geq s\geq t\geq 0\}$. 
\end{defn}
The set of triples $\mathcal G_d$ corresponds to directed abelian Cayley graphs with specified robber movesets, and the set of triples $\mathcal G_u$ corresponds to undirected abelian Cayley graphs with specified robber movesets. The set $\mathcal D$ then includes all possible sizes for a triple in $\mathcal G_d$ or $\mathcal G_u$ (along with some unattainable triple sizes).

\newpage

\hspace{10mm}

We also define the following.

\begin{defn}
    Let $(n,s,t) \in \mathcal D$. 
    \begin{itemize}
        \item We define $c_d(n,s,t)$ as the maximum, over all triples $(G,S,T) \in \mathcal G_d$ of respective sizes $(n,s,t)$, of the number of cops required to capture a robber on $\cay(G,S)$ when the robber may only play moves in $T$.
        \item We define $c_u(n,s,t)$ as the maximum, over all triples $(G,S,T) \in \mathcal G_u$ of respective sizes $(n,s,t)$, of the number of cops required to capture a robber on $\cay(G,S)$ when the robber may only play moves in $T$.
      \end{itemize}
\end{defn}

Whenever there exists no triple $(G,S,T) \in \mathcal G_d$ of respective sizes $(n,s,t) \in \mathcal D$, we say that $c_d(n,s,t) = 1$. Similarly, whenever there exists no triple $(G,S,T) \in \mathcal G_u$ of respective sizes $(n,s,t)$ for some triple $(n,s,t) \in \mathcal D$, we say that $c_u(n,s,t) = 1$. Furthermore, $c_d(n,s,t)$ and $c_u(n,s,t)$ have a trivial upper bound of $n$. Thus, $c_d(n,s,t)$ and $c_u(n,s,t)$ are well-defined for all $(n,s,t) \in \mathcal D$.

Note that as $\mathcal G_u \subseteq \mathcal G_d$, it immediately follows that for any triple $(n,s,t) \in \mathcal D$, $c_u(n,s,t) \leq c_d(n,s,t)$. Furthermore, in a standard game of cops and robbers, the robber may choose any move in $S$ on each turn, and hence for a finite abelian group $G$ generated by set $S$, the cop number of $\cay(G,S)$ is at most $c_d(n,s,s)$ in general, and the cop number of $\cay(G,S)$ is at most $c_u(n,s,s)$ when $S = -S$.

For technical reasons which will become clear shortly, we also need to define the following.

\begin{defn}
      We define $\mathcal B = \{(n,s,t)\in \mathcal D : t=0 \text{ or } s=n-1\}$. We say that triples $(n,s,t) \in\mathcal B$ are \emph{boundary values.}
\end{defn}

In other words, boundary values give the sizes of triples $(G,S,T) \in \mathcal G_d$ for which determining the number of cops required to capture a robber on $\cay(G,S)$ with moveset $T$ is trivial, as we see in the following observation.

\begin{lem}\label{admissiblelemma}
If $(n,s,t) \in \mathcal B$, then $c_d(n,s,t)= c_u(n,s,t) =1$.
\end{lem}
\begin{proof}
For a triple $(n,s,t) \in \mathcal B$, let $(G,S,T) \in \mathcal G_d$ such that $(|G|,|S|,|T|) = (n,s,t)$. (If no such triple $(G,S,T)$ exists, then $c_d(n,s,t) = c_u(n,s,t) = 1$ by definition.) Consider a game of cops and robbers on $\cay(G,S)$ in which the robber's moveset is $T$.
\begin{itemize}
\item If $T = \emptyset$, then the robber has no moves, and a single cop may move to the robber's position and capture the robber.

\item If $|S| = |G| - 1$, then $\cay(G,S)$ is a complete digraph, and a single cop can capture the robber after one move.

\end{itemize}
Thus, we have shown that $c_d(n,s,t) = 1$. As $1 \leq c_u(n,s,t) \leq c_d(n,s,t)$, it also follows that $c_u(n,s,t) = 1$.
\end{proof}
We note that for a triple $(G,S,T) \in \mathcal G_d$ of respective sizes $(n,s,t)$, if $n \leq 2$, then $s = n-1$ must hold, so $(n,s,t) \in \mathcal B$. Therefore, when we consider values $(n,s,t) \in \mathcal D \setminus \mathcal B$, we may assume that $n \geq 3$.

We are now ready for our main tool, which will be the following lemma. This lemma essentially formalizes a general inductive strategy of capturing the robber by guarding robber moves until no robber move is safe. The lemma generalizes key ideas used by Frankl in \cite{FranklCayley} and uses the idea from \cite{PB} of having some elements of $S$ which account for many elements of $T$.

We first give an informal description of the lemma. We will have functions $g$ and $h$ taking values in $\mathcal D \setminus \mathcal B$. The function $h$ will give a lower bound for the number of robber moves that a single cop can make unsafe on an $n$-vertex $s$-regular abelian Cayley graph or digraph, when the robber's moveset is of size $t$. We will show that if $g$ satisfies certain properties that are necessary for the inductive strategy of guarding robber moves described above, then $g(n,s,t)$ gives an upper bound for the number of cops needed to capture a robber on such a graph.

\begin{lem}\label{mainlemma}
    Let $(\mathcal G,c)$ be either $(\mathcal G_d,c_d)$ or $(\mathcal G_u,c_u)$, and let $g:\mathcal D\setminus \mathcal B \rightarrow \R^{\geq 2}$ and $h: \mathcal D\setminus \mathcal B\rightarrow \R^{>0}$ be functions. Suppose that $g$ and $h$ respect the following conditions for all $(n,s,t)\in \mathcal D\setminus \mathcal B$:
	\begin{enumerate}
		\item For any $(G,S,T)\in \mathcal G$ with respective sizes $(n,s,t) \in \mathcal D \setminus \mathcal B$, there exists an element $k \in G \setminus \{0_G\}$ accounting for at least $h(n,s,t)$ elements of $T$ with respect to $S$;
		\item For $n'\leq \frac{n}{2}$, $s'\leq s$, $t'\leq t$, and $(n',s',t')\in \mathcal D\setminus \mathcal B$, either $g(n,s,t)\geq c(n',s',t')$, or $g(n,s,t)\geq g(n',s',t')$;
		\item If $1 \leq t'\leq t-h(n,s,t)$, then $g(n,s,t)\geq g(n,s,t')+1$.
	\end{enumerate}	
	
	 Then, if $(n,s,t)\in \mathcal D \setminus \mathcal B$, then $c(n,s,t)\leq g(n,s,t)$.
\end{lem}
\begin{proof}
We fix functions $g$ and $h$ that satisfy the conditions of the lemma.

Suppose that the lemma does not hold for some $(n,s,t) \in \mathcal D \setminus \mathcal B$. We choose our offending triple $(n,s,t)$ with $n$ as small as possible and, subject to $n$ being minimum, with $t$ as small as possible. As the lemma does not hold for $(n,s,t)$, we may choose $(G,S,T) \in \mathcal G$ with respective sizes $(n,s,t)$ so that $g(n,s,t)$ cops are not enough to capture a robber on $\cay(G,S)$, even when the robber may only play moves from $T$. We will show that this gives us a contradiction.

By condition (1), there exists an element $k \in G$, satisfying $k\neq 0_G$, that accounts for at least $h(n,s,t)$ elements of $T$. We would like to show that we can position a cop at a vertex $r + \gamma k$, where $r$ is the position of the robber, and $\gamma$ is some integer. In other words, we would like to show that we can capture the robber ``modulo $k$." To this end, we let $\phi: G \rightarrow G / \langle k \rangle$ be the natural homomorphism $a \mapsto a +  \langle k \rangle$. By the definition of $\phi$, placing a cop at such a vertex $r + \gamma k$ is equivalent to capturing the robber in a game of cops and robbers played on $G / \langle k \rangle$ with cop moveset $\phi(S)$ and robber moveset $\phi(T)$.

We first note that $(G / \langle k \rangle,\phi(S),\phi(T))\in \mathcal G$. In particular, if $S=-S$, then $\phi(S)=\phi(-S)=-\phi(S)$.

We now show that our $g(n,s,t)$ cops have a strategy to capture the robber in the game on $G / \langle k \rangle$. 
As $k \neq 0_G$, we see that $n'=|G / \langle k \rangle| \leq n/2$, $s'=|\phi(S)| \leq |S|=s$, and $t'=|\phi(T)| \leq |T|=t$. If $(n',s',t')\in \mathcal B$, then as shown previously, $c(n',s',t')=1 < g(n,s,t)$, and our $g(n,s,t)$ cops may capture the robber on $G / \langle k \rangle$. Otherwise, suppose that $(n',s',t')\in \mathcal D\setminus \mathcal B$. Then $c(n',s',t') \leq g(n,s,t)$ either directly from (2), or from the inequality $c(n',s',t') \leq g(n',s',t') \leq g(n,s,t)$, which follows from the fact that $(n,s,t)$ is a minimal counterexample. Therefore, our $g(n,s,t)$ cops have a strategy by which a cop $C$ can reach a vertex $r+\gamma k$ for some integer $\gamma \geq 0$, where $r \in G$ is the position of the robber. 

Next, we show that at this point, $C$ has a strategy to restrict the robber to a moveset of size at most $t - h(G,S,T)$. Let $A = \{a_1, \dots, a_m\} \subseteq T$ be the set of robber moves accounted for by $k$. If the robber plays a move $a' \not \in A$, then $C$ plays $a'$, and $C$ will stay at a vertex of the form $r+\gamma k$, where $r$ is the new position of the robber. If the robber plays a move $a_i \in A$, then $C$ has a move $b_i \in S\cup\{0_G\}$ such that $a_i-b_i = k$. After $C$ plays $b_i$, $C$ now occupies a vertex $r+(\gamma-1)k$, where $r$ is the new position of the robber. Thus we see that whenever the robber plays a move from $A$, which must be accounted for by $k$, the ``difference" between the robber and $C$ decreases by exactly $k$. Thus, if the robber plays a move from $A$ sufficiently many times ($\gamma$ times), then the robber will be caught by $C$. Therefore, the robber must eventually stop playing all moves of $A$. The number of moves in $A$ is at least $h(n,s,t)$, and hence $C$ restricts the robber to a moveset $T \setminus A$ of size at most $t - h(n,s,t)$.

We note that when applying the inductive strategy on the quotient graph $\cay(G / \langle k \rangle, \phi(S))$, it is still possible for the robber to play moves which are not considered safe, but only a bounded number of times. For example, in the paragraph above, we describe the move $a_i$ as unsafe, but the robber may play $a_i$ up to $\gamma - 1$ times without being captured. If the robber plays an ``unsafe" move, we pause the inductive strategy; then all cops playing the quotient strategy copy the robber's move, while the cops guarding this unsafe move advance closer to the robber.

Now, we show that it is possible for at most $g(n,s,t)-1$ cops to win in the game given by the triple $(G,S,T\setminus A)$. This will give us our contradiction, as we may then capture the robber in the game given by $(G,S,T)$ with $g(n,s,t)$ cops by using one cop to make the moves in $A$ unsafe for the robber and then using the remaining $g(n,s,t)-1$ cops to win in $(G,S,T\setminus A)$. If $(n,s,t - |A|)$ is a boundary value, then 1 cop is sufficient for the game given by $(G,S,T\setminus A)$, and then as $g(n,s,t)\geq 2$, we have our contradiction. Note that as $(n,s,t) \in \mathcal D \setminus \mathcal B$, $(n,s,t - |A|) \in \mathcal B$ if and only if $t - |A| = 0$. Otherwise, $t - |A| \geq 1$ and $(n,s,t - |A|)\in \mathcal D\setminus \mathcal B$, and by the minimality of $(n,s,t)$ and condition (3), $g(n,s,t - |A|)-1$ additional cops are sufficient to capture the robber in the game given by the triple $(G,S,T\setminus A)$. Therefore, in total, we need at most $g(n,s,t)$ cops to capture the robber, which contradicts the minimality of $(n,s,t)$. Thus, our proof is complete.
 \end{proof}

\section{Upper bound for undirected abelian Cayley graphs}\label{undirectedsection}
In this section, we will show that the approach we have outlined in Lemma \ref{mainlemma} gives us an upper bound of $\frac{1}{\sqrt{(\sqrt{2}-1) e}} \sqrt{n}+\frac{5}{2}\approx 0.9424\sqrt{n}+\frac{7}{2}$ on the cop number of undirected abelian Cayley graphs of $n$ vertices. As we consider undirected graphs in this section, whenever we have an abelian group $G$ generated by a set $S$, we will always require that $S = -S$. This way, we may consider $\cay(G,S)$ as an undirected abelian Cayley graph. Some of the symbolic and optimization computations in this section and the next were done with Mathematica \cite{Mathematica}, but with care and patience, each computation can be checked by hand.

Using our main tool of Lemma \ref{mainlemma}, we will aim to define functions $g$ and $h$ that satisfy its conditions for $(\mathcal 
G_u,c_u)$ and such that $g(n,s,s)$, which is an upper bound for cop number, is not too large. Therefore, the main challenge of this section will be choosing suitable functions $g$ and $h$, and moreover, showing that these functions satisfy all the conditions of Lemma \ref{mainlemma}. 

Before we seek our functions $g$ and $h$ with which to prove an upper bound on the cop number of abelian Cayley graphs, we first note that there is a simple choice of $g$ and $h$ that gives a slightly weaker version of Theorem \ref{thmFrankl}, as shown in the following proposition. This simple choice of the functions $g$ and $h$ gives an instructive example of how to use Lemma \ref{mainlemma}, and this application of Lemma \ref{mainlemma} furthermore shows that our lemma is indeed a generalization of the method of Frankl from \cite{FranklCayley}.

\begin{prop}\label{weakFrankl}
If $\Gamma$ is a Cayley graph on an abelian group with generating set $S$ such that $S=-S$ and $0_G\notin S$, then
$$c(\Gamma)\leq \left\lceil\frac{|S|+3}{2}\right\rceil.$$
\end{prop}

\begin{proof}
We wish to define functions $g$ and $ h$ taking values in $\mathcal D \setminus \mathcal B$ that satisfy the conditions of Lemma \ref{mainlemma} for $(\mathcal G_u,c_u)$. We choose
	$$g(n,s,t)=\left\lceil\frac{t+3}{2} \right\rceil;$$
$$h(n,s,t)=
\begin{cases}
    1 & t=1,2\\
    2 & t\geq 3.
\end{cases}
$$
We claim that $g$ and $h$ satisfy the conditions of Lemma \ref{mainlemma} for $(\mathcal G_u,c_u)$. We immediately notice that $h(n,s,t) > 0$ by definition, and for all $(n,s,t) \in \mathcal D \setminus \mathcal B$, we have $t \geq 1$, and hence $g(n,s,t) \geq 2$.

We show that condition (1) of Lemma \ref{mainlemma} is satisfied. Indeed, let $(G,S,T)\in \mathcal G_u$ with respective sizes $(n,s,t) \in \mathcal D \setminus \mathcal B$. If $t \in \{1,2\}$, then for any $a \in T$, we may let $k = a$; then, as $a - 0_G = k$, $k$ accounts for at least $1$ element of $T$ with respect to $S$, namely $a$. If $t \geq 3$, then we may choose any elements $a,b \in T$ with $a \neq -b$ and let $k = a + b$; then, as $a - (-b) = k$ and $b - (-a) = k$, $k$ accounts for at least two elements of $T$ with respect to $S$, namely $a$ and $b$. Hence, condition (1) of Lemma \ref{mainlemma} is satisfied.

Next, we show that condition (2) of Lemma \ref{mainlemma} is satisfied. If we have $n' \leq n/2$, $s' \leq s$, and $t' \leq t$ such that $(n',s',t')  \in \mathcal D \setminus \mathcal B$, then $$g(n,s,t) = \left\lceil\frac{t+3}{2} \right\rceil \geq \left\lceil\frac{t'+3}{2} \right\rceil = g(n',s',t').$$

Finally, we show that condition (3) of Lemma \ref{mainlemma} is satisfied. Suppose $1 \leq t'\leq t-h(n,s,t)$. As $h(n,s,t) \geq 1$, we must have $t \geq 2$. If $t = 2$, then $t' = 1$, and $g(n,s,t) = 3 = 2+1 = g(n,s,t') + 1$. Otherwise, if $t \geq 3$, then $h(n,s,t) = 2$. Thus, if $1 \leq t' \leq t - 2$, then 
$$g(n,s,t) =  \left \lceil \frac{t + 3}{2} \right \rceil \geq \left \lceil \frac{t' + 3}{2} \right \rceil + 1 = g(n,s,t') + 1.$$

Hence, as $g$ and $h$ satisfy the conditions of Lemma \ref{mainlemma}, it follows that $c_u(n,s,t) \leq g(n,s,t)$ for all $(n,s,t) \in \mathcal D\setminus \mathcal B$. Therefore, for an abelian group $G$ of $n$ elements generated by a set $S \subseteq G$ of $s$ elements satisfying $S = -S$, if $\Gamma = \cay(G,S)$, one of the following holds: either $(n,s,s) \in \mathcal B$ and $c(\Gamma) \leq c_u(n,s,s) = 1 < \lceil \frac{s+3}{2}\rceil$, or $(n,s,s) \in \mathcal D \setminus \mathcal B$ and $c(\Gamma) \leq c_u(n,s,s) \leq g(n,s,s) = \lceil \frac{s+3}{2}\rceil$. 
\end{proof}

In fact, by defining boundary values more carefully, it is possible to obtain the exact result of Theorem \ref{thmFrankl} from Lemma \ref{mainlemma}. More precisely, if we add the inductive base cases for Frankl's proof from from \cite{FranklCayley} as boundary values in the undirected case, then we may use Lemma \ref{mainlemma} to give the exact same result as Theorem \ref{thmFrankl}. However, this modification requires that boundary values be defined separately for the directed and undirected cases, and it adds to the already existing technicalities, so we opt for a simpler presentation with a slightly worse additive constant.

In the proof of Theorem \ref{thmFrankl}, we let a single element of $G$ account for at most two elements of $T$. As discussed earlier, we will see that in general, a single element can account for many more than two elements of $T$. This will allow us to use a similar strategy to find an improved upper bound for the cop number of an abelian Cayley graph. 

For the remainder of this section, our goal will be to establish a sharper upper bound on the cop number of an undirected Cayley graph on an abelian group. The main tool for our improved upper bound will be Lemma \ref{mainlemma}, so as discussed before, we will seek functions $g$ and $h$ that we can use with Lemma \ref{mainlemma}. 
In the following definition, we define a function $h$ that we will use for the entire remainder of this section. In the definition of $h$, we will use a fixed constant $c > 0$. We will assign a value to $c$ later.

\begin{defn}
    	We define the function $h:\mathcal D\setminus \mathcal B\rightarrow \R^{> 0}$ by
	$$h(n,s,t)=
\begin{cases}
1&t\in\{1,2\}\ \textrm{ and } t\leq c\sqrt{n}\\
2&3 \leq t\leq c\sqrt{n}\\
\frac{ts}{n-1}&t>c\sqrt{n}.
\end{cases}$$
\end{defn}

We note that it is important to add that $t\leq c\sqrt{n}$ in the first condition. We will sometimes choose $c$ to be as low as about $0.8$, so when $n$ is small, it is possible that $c \sqrt{n} < 2$.

\begin{lem}\label{undirectedfunction}
The function $h$ satisfies condition (1) of Lemma \ref{mainlemma} for $(\mathcal G_u, c_u)$.
\end{lem}

\begin{proof}
	We must show that if $(G,S,T) \in \mathcal C_u$ is a triple with respective sizes $(n,s,t) \in \mathcal D \setminus \mathcal B$, then there exists an element $k \in G \setminus \{0_G\}$ accounting for at least $h(n,s,t)$ elements of $T$ with respect to $S$.

    Suppose that $t \leq c \sqrt{n}$. Then $h$ is defined as in the proof of Theorem \ref{weakFrankl} given above, and the statement follows from the same argument.
    
    Suppose, on the other hand, that $t>c\sqrt{n}$.
	 We compute a multiset $M$ consisting of all differences $a_i - a_j$, for $a_i \neq a_j$, $a_i\in T$, and $a_j\in S\cup\{0_G\}$. Let $k$ be a most frequently appearing element of $M$. There are $t$ possible choices for $a_i$, and there are $s$ possible choices for each $a_j$, namely $0_G$ and every element of $S \setminus \{a_i\}$.
	 
	 By the pigeonhole principle, as each element of $M$ is one of $n-1$ possible values, the most commonly occurring element $k$ of $M$ must appear at least $\frac{ts}{n-1}$ times. Therefore, $k$ must account for at least $\frac{ts}{n-1}$ elements of $T$ with respect to $S$. As $k \neq 0_G$, the statement again holds.
\end{proof}
	
	The cutoff at $c \sqrt{n}$ in the definition of $h$ is analogous to the cutoff in the Pairing Algorithm of \cite{PB}. The idea behind this cutoff is the fact that when $t$ is small, the quantity $\frac{ts}{n-1}$ becomes smaller than $2$, and then it is preferable to argue directly that there exists an element of $G \setminus \{0_G\}$ accounting for two elements of $T$. As discussed in Section \ref{furtherdirections}, we could, of course, take the ceiling of this function when applying the pigeonhole principle, but this makes analysis of the function $h$ very difficult. One might also be interested in modifying this cutoff to be of the form $t>c \frac{n-1}{s}$ (for some constant $c\geq 1$) in order for $h$ always to be at least 1. For some triples $(n,s,t)$, this would allow $h(n,s,t)$ to become larger while still satisfying condition (1). However, this alternative cutoff of $h$ does not appear to behave nicely when it comes to verifying condition (2).
	
	 Next, we will define our function $g$. Our goal for the remainder of this section will then be to show that $g$ and $h$ satisfy the conditions of Lemma \ref{mainlemma} for $(\mathcal G_u, c_u)$ and that $g(n,s,s)$ is not too large. 
	
	\begin{defn}
	We define the function $g:\mathcal D\setminus \mathcal B\rightarrow \R^{\geq 2}$ by
	
$$g(n,s,t)=\begin{cases}
\left\lceil\frac{t+3}{2} \right\rceil & 1\leq t\leq c\sqrt{n} \\
\frac{\log \frac{t}{c\sqrt n}}{\log \frac{n-1}{n-s-1}}+\frac{c\sqrt n}{2} +\frac{7}{2}   &    t>c\sqrt{n}.
\end{cases}$$ 
\end{defn}

This choice of $g$ may not seem straightforward, so we present the intuition behind this definition. We suppose that for a value $(n,s,t) \in \mathcal D\setminus \mathcal B$, we have an abelian group $G$ on $n$ elements generated by a set $S \subseteq G$  (with $S= -S$) of $s$ elements, and a subset $T \subseteq S$ of $t$ elements. We would like to estimate the number of elements of $G$ needed to form a set $K$ such that the elements of $K$ altogether account for each element of $T$, since, as we have discussed, this will help us count the number of cops needed to make every robber move unsafe in a game on $\cay(G,S)$.

In order to estimate the number of elements needed in $K$, we may construct $K$ iteratively. The iterative construction that we describe here is a refinement of the Pairing Algorithm from \cite{PB}. If $t \leq c\sqrt{n}$, then we may pair the elements of $T$ by the method of Proposition \ref{weakFrankl} and obtain a set $K$ of at most $\lceil \frac{t+3}{2} \rceil$ elements, for which the elements of $K$ altogether account for all of $T$ with respect to $S$. On the other hand, if $t > c\sqrt{n}$, we may choose one element $k \in G$ to account for at least $\frac{ts}{n-1}$ elements of $T$, as in the proof of Lemma \ref{undirectedfunction}. More generally, we can use the same idea to define a recursive process that repeatedly adds elements to $K$, and we may run this process until at most $c\sqrt{n}$ elements of $T$ are not accounted for by $K$. We execute our recursive process as follows. We define $z_i$ to be the number of elements accounted for by $K$ after $i$ iterations of our process. We immediately see that $z_0=0$, and if we choose $k$ as described earlier during the first iteration of our process, we may let $z_1 \geq \frac{ts}{n-1}$. Additionally, given $z_{i-1}$, there are $t-z_{i-1}$ elements of $T$ not accounted for by $K$, and hence on the $i$th iteration of our procedure, we may add an element to $K$ that accounts for $\frac{s(t-z_{i-1})}{n-1}$ new elements of $T$. Therefore, we obtain a recursive inequality for the number of elements in $T$ accounted for by $K$ after $i$ iterations of our procedure: $$z_i\geq z_{i-1}+\frac{s(t-z_{i-1})}{n-1}=\frac{n-s-1}{n-1}z_{i-1}+\frac{st}{n-1},$$
which has a closed form of
$$z_i\geq t-t \left(\frac{n-s-1}{n-1}\right)^i.$$

Hence, after $i$ iterations of our recursive procedure, we calculate that there are at most $t \left(\frac{n-s-1}{n-1}\right)^i$ elements of $T$ not accounted for by $K$. As soon as the number of elements in $T$ not accounted for by $K$ is at most $c\sqrt{n}$, we may pair the remaining elements of $T$ as in the proof of Proposition \ref{weakFrankl}. Therefore, the recursive method we have described will run $i$ times, where $i$ is the smallest integer such that $t \left(\frac{n-s-1}{n-1}\right)^i\leq c\sqrt{n}$. We thus may calculate that 
$$i =\left\lceil\frac{\log\frac{t}{c\sqrt{n}}}{\log\frac{n-1}{n-s-1}}\right\rceil$$
and hence after the recursive method runs $i$ times, at most $c \sqrt{n}$ elements of $T$ will be left unaccounted for by $K$. At this point, the remaining $c \sqrt{n}$ unaccounted elements of $T$ may be paired into sums, as in the method of Theorem \ref{weakFrankl}, and each such sum roughly accounts for $2$ elements of $T$. We may put these sums into $K$, at which point the elements of $K$ altogether account for all of $T$. In total, our count shows that our set $K$ needs roughly at most 
$$\left\lceil \frac{\log\frac{t}{c\sqrt{n}}}{\log\frac{n-1}{n-s-1}} \right\rceil + \frac{c\sqrt{n}}{2}$$
elements. This counting method gives us an intuition with which we define the function $g$. The extra additive constant of $g$ is included for technical reasons that will become clear later.

We easily see that our function $g$ is defined on all values in $\mathcal D \setminus B$ and bounded below by $2$. In the following lemmas, we will bound $g(n,s,t)$ above, and we will show that $g$ and $h$ satisfy conditions (2) and (3) of Lemma \ref{mainlemma}.

\begin{lem}\label{undirectedbound}
	Let $d > 0$. If $d\geq \frac{1}{ce}+\frac{c}{2}$, then $g(n,s,t)\leq d\sqrt{n}+\frac{7}{2}$.
\end{lem}
\begin{proof}
	We consider two cases :
	\begin{enumerate}
		\item If $t\leq c\sqrt{n}$, then
			$$g(n,s,t)=\left\lceil \frac{t+3}{2}\right\rceil\leq \left\lceil \frac{c\sqrt{n}+3}{2}\right\rceil\leq \frac{c}{2}\sqrt{n}+\frac{5}{2} < d\sqrt{n}+\frac{7}{2}.$$
		\item  If $t>c\sqrt{n}$, then we first note that $g(n,s,t)\leq g(n,s,s)$. We wish to find $\alpha$ such that $\frac{\log \frac{s}{c\sqrt n}}{\log \frac{n-1}{n-s-1}}\leq \alpha \sqrt{n}$ for all real values $n \geq 3$ and $1 \leq s \leq n - 2$.
		This inequality can be rewritten as
\begin{equation} \tag{$*$}
\label{inequality}
 r_{\alpha,c}(n,s):= \frac{c\sqrt{n}}{s}\left(\frac{n-1}{n-s-1}\right)^{\alpha \sqrt{n}} \geq 1.
		\end{equation}

		One calculates that the derivative relative to $s$ is
		$$\frac{\partial r_{\alpha,c}}{\partial s}=-\frac{c \sqrt{n} \left(\frac{n-1}{n-s-1}\right)^{\alpha \sqrt{n}+1} \left(-\alpha \sqrt{n} s+n-s-1\right)}{(n-1) s^2}.$$
		
		By examining the sign of this derivative, we see that $r_{\alpha,c}(n,s)$ achieves a minimum when $s$ has a value $s^*=\frac{n-1}{\alpha \sqrt{n}+1}$.
		Therefore, it will suffice to choose a value $\alpha$ such that the inequality (\ref{inequality}) holds when $s$ is replaced by $s^*$. We will show that the choice $\alpha=\frac{1}{c e}$ works. 

		By substituting $s$ with $s^* = \frac{n-1}{\alpha \sqrt{n}+1}$, applying $\alpha = \frac{1}{ce}$, and doing some simplification, we find that 
		$$w_c(n):=r_{\frac{1}{ce},c}(n,s^*)=\frac{ \left(1 + \frac{ce}{\sqrt{n}}\right)^{\frac{\sqrt{n}}{ce}} \sqrt{n} \left(ce+\sqrt{n}\right)}{e (n-1)}.$$
		In order to show that the inequality (\ref{inequality}) holds with $\alpha = \frac{1}{ce}$, it will be enough to show that $w_c(n) \geq 1$.

		We will apply the inequality $(1+\frac{x}{y})^y>e^{\frac{xy}{x+y}}$ (for $x,y>0$) \cite{inequalitylist} \cite[Section 5.3]{Kuang} with $x = \frac{1}{\sqrt{n}}$, $y = \frac{1}{ce}$, along with the inequality $\frac{\sqrt{n}}{n-1} > \frac{1}{\sqrt{n}}$. With these two inequalities, we find
		\begin{align*}
		    w_c(n)&=\frac{ \left(1+\frac{ce}{\sqrt{n}}\right)^{\frac{\sqrt{n}}{ce}} \sqrt{n} \left(ce+\sqrt{n}\right)}{e (n-1)}
		    >\frac{ e^{\frac{\sqrt{n}}{c e+\sqrt{n}}} \sqrt{n} \left(ce+\sqrt{n}\right)}{e (n-1)} >  \frac{ e^{\frac{-ce}{c e+\sqrt{n}}} \left(ce+\sqrt{n}\right)}{ \sqrt{n}}=:z_c(n).
		\end{align*}
		
		Furthermore, one calculates that
		$$z_c'(n)=-\frac{c^2 e^{2-\frac{ce}{ce+\sqrt{n}}}}{2 \left(ce n^{3/2}+n^2\right)}$$
		which is always negative, and that
		$$\lim_{n\rightarrow\infty} z_c(n)=1.$$

		Therefore, $w_c(n) > z_c(n) >  1$, which confirms that $\frac{\log \frac{s}{c\sqrt n}}{\log \frac{n-1}{n-s-1}}\leq \alpha \sqrt{n}$ when $\alpha = \frac{1}{ce}$, for all real values $c > 0$, $n \geq 3$, $1 \leq s \leq n-2$. Thus, we have
		$$g(n,s,t)=\frac{\log \frac{t}{c\sqrt n}}{\log \frac{n-1}{n-s-1}}+\frac{c\sqrt n}{2}+\frac{7}{2}\leq \left(\frac{1}{ce}+\frac{c}{2}\right)\sqrt{n}+\frac{7}{2}\leq d\sqrt{n}+\frac{7}{2}.$$
	\end{enumerate}
\end{proof}

\begin{lem}\label{undirectedcondition2}
	Let $d > 0$. If 
	$\frac{c}{2}\geq \frac{d}{\sqrt 2}$ and $d\geq \frac{1}{ce}+\frac{c}{2}$, then
	$g$ respects condition (2) of Lemma \ref{mainlemma} for $(\mathcal G_u, c_u)$.
\end{lem}

\begin{proof}
	Consider a choice of $d>0$ such that $\frac{c}{2}\geq \frac{d}{\sqrt 2}$ and $d\geq \frac{1}{ce}+\frac{c}{2}$. Let $(n,s,t)\in \mathcal D\setminus \mathcal B$ and $(n',s',t')\in \mathcal D \setminus \mathcal B$ such that $n'\leq \frac{n}{2}$, $s'\leq s$ and $t'\leq t$. We consider two cases:
	\begin{enumerate}
		\item If $t\leq c\sqrt n$, then $$g(n,s,t)=\left\lceil\frac{t+3}{2}\right\rceil\geq \left\lceil\frac{t'+3}{2}\right\rceil\geq c_u(n',s',t'),$$
		using a result from the proof of Proposition \ref{weakFrankl}.
		\item If $t>c\sqrt{n}$, then by the previous lemma and our hypotheses on $c$ and $d$,
		\begin{align*}
			 g(n,s,t)=\frac{\log \frac{t}{c\sqrt n}}{\log \frac{n-1}{n-s-1}}+\frac{c\sqrt n}{2} + \frac{7}{2}>  \frac{c\sqrt{n}}{2}+\frac{7}{2} \geq  d\frac{\sqrt{n}}{\sqrt{2}}+\frac{7}{2}\geq  d\sqrt{n'}+\frac{7}{2}\geq g(n',s',t').
		\end{align*}
	\end{enumerate}
\end{proof}

\begin{lem}\label{undirectedcondition3}
	The functions $g$ and $h$ respect condition (3) of Lemma \ref{mainlemma} for $(\mathcal G_u, c_u)$.
\end{lem} 

\begin{proof}
Let $(n,s,t)\in \mathcal D\setminus \mathcal B$ and $1\leq t'\leq t-h(n,s,t)$.

We consider four cases :
	\begin{enumerate}
	    \item If $t\leq 2$ and $t\leq c\sqrt{n}$, then as $h(n,s,t)=1$ we can see that the only case to consider is $t=2$ and $t'=1$. In this case, $g(n,s,t)=3$ and $g(n,s,t')=2$, so $g(n,s,t) = g(n,s,t') + 1$.
		\item If $3 \leq t\leq c\sqrt{n}$, then $h(n,s,t)= 2$, and thus $t\geq t'+2$. Then,
		$$g(n,s,t)=\left\lceil\frac{t+3}{2}\right\rceil\geq \left\lceil\frac{t'+3}{2}\right\rceil=\left\lceil\frac{t'+1}{2}\right\rceil+1= g(n,s,t') + 1$$
		\item If $t> c\sqrt{n}$ and $t'\leq  c\sqrt{n}$, then
			\begin{align*}
				g(n,s,t)=\frac{\log \frac{t}{c\sqrt n}}{\log \frac{n-1}{n-s-1}}+\frac{c\sqrt n}{2} + \frac{7}{2} >  \frac{c\sqrt n+3}{2}+2 > \left\lceil \frac{c\sqrt n+3}{2}\right\rceil+1\geq \left\lceil \frac{t'+3}{2}\right\rceil+1 \geq g(n,s,t') + 1
			\end{align*}
		\item If $t,t'>c\sqrt{n}$, we know from Lemma \ref{undirectedfunction} that $t'\leq t-\frac{ts}{n-1}=t\left(\frac{n-s-1}{n-1}\right)$. Thus,
				$$
				g(n,s,t)=\frac{\log \frac{t}{c\sqrt n}}{\log \frac{n-1}{n-s-1}}  + \frac{c\sqrt n}{2} + \frac{7}{2}  \geq \frac{\log \left(\frac{t'}{c\sqrt n}\cdot  \frac{n-1}{n-s-1}\right)}{\log \frac{n-1}{n-s-1}}  + \frac{c\sqrt n}{2} + \frac{7}{2}=\frac{\log \left(\frac{t'}{c\sqrt n} \right)}{\log \frac{n-1}{n-s-1}}+1  + \frac{c\sqrt n}{2} + \frac{7}{2}=g(n,s,t')+1.$$ 
	\end{enumerate}
	In all cases, $g(n,s,t) \geq g(n,s,t') + 1$, so condition (3) of Lemma \ref{mainlemma} is satisfied for $(\mathcal G_u,c_u)$.
\end{proof}

We have shown that $g$ and $h$ satisfy the conditions of Lemma \ref{mainlemma}, so we are ready for our main result for undirected Cayley graphs on abelian groups.
\begin{thm}\label{undirectedupperbound}
	The cop number of any undirected Cayley graph on an abelian group of $n$ elements is at most $\frac{1}{\sqrt{(\sqrt{2}-1) e}} \sqrt{n}+\frac{7}{2}\approx 0.9424\sqrt{n}+\frac{7}{2}$.
\end{thm}

\begin{proof}
Let $G$ be an abelian group on $n$ vertices generated by set $S \subseteq G$ (with $S = -S$ and $ 0_G\notin S$) of $s$ elements. If $(n,s,s)\in \mathcal B$, then the result follows directly from Lemma \ref{admissiblelemma}. Otherwise, we assume that $(n,s,s)\in \mathcal D\setminus B$.

We first find values $c$ and $d$ satisfying $\frac{c}{2}\geq \frac{d}{\sqrt 2}$ and $d\geq \frac{1}{ce}+\frac{c}{2}$, which minimize $d$. A simple computation of such values $c$ and $d$ yields $c=\sqrt{\frac{2}{e(\sqrt{2} -1)}}\approx 1.33$ and $d=\sqrt{\frac{1}{e(\sqrt{2} -1)}} \approx 0.9424$.

With these chosen values of $c$ and $d$, the lemmas of this section show that our choices for $g$ and $h$ satisfy all three conditions of Lemma \ref{mainlemma} for the case $(\mathcal G_u, c_u)$. Hence, by Lemmas \ref{mainlemma} and \ref{undirectedbound}, $c(\cay(G,S)) \leq c_u(n,s,s) \leq g(n,s,s)\leq d\sqrt{n}+\frac{7}{2}$.
\end{proof}

Similarly to Proposition \ref{weakFrankl}, the additive constant of Theorem \ref{undirectedupperbound} can be improved by $1$ to $\frac{5}{2}$ with a more technical definition of boundary values. However, we do not feel that this slight improvement justifies the added technicalities.

We note that Theorem \ref{undirectedupperbound} not only proves that Meyniel's conjecture holds for undirected Cayley graphs on abelian groups (with a smaller multiplicative constant than in \cite{PB}), but it also proves that Meyniel's conjectured bound holds for these graphs with a coefficient of $\sqrt{n}$ smaller than $1$. Indeed, Wagner has conjectured in \cite{wagner} that the coefficient of $\sqrt{n}$ in Meyniel's conjecture should be $1$, so Theorem \ref{undirectedupperbound} shows that abelian Cayley graphs satisfy both the conjectured upper bounds of Meyniel and Wagner.

In the next proposition, we show that we may obtain marginal improvements on the coefficient of $\sqrt{n}$ by considering the group structure of $G$. The proposition uses the fact that for a group $G$ of $n$ elements such that the smallest prime divisor of $n$ is a prime $p$, no element of prime order $q < p$ exists in $G$, and moreover, by Lagrange's Theorem and prime factorization, no such element exists in any subgroup or quotient group of $G$.

\begin{prop}\label{undirectedprime}
	Let $G$ be an abelian group on $n$ of vertices, and let $S \subseteq G$ be a generating set of $G$ such that $S = -S$ and $0_G\notin S$. Let $p$ be the smallest prime factor of $n$.
	\begin{enumerate}
		\item If $p=3$, then $c(\cay(G,S))\leq \sqrt{\frac{3}{2 \left(\sqrt{3}-1\right) e}}\sqrt{n}+\frac{7}{2}\approx 0.8682\sqrt{n}+\frac{7}{2}$.
		\item If $p\geq 5$, then $c(\cay(G,S))\leq\sqrt{\frac{2}{e}}\sqrt{n} +\frac{7}{2}\approx 0.8578\sqrt{n}+\frac{7}{2}$.
	\end{enumerate}
\end{prop}
\begin{proof}

    \begin{enumerate}
        \item In condition (2) of Lemma \ref{mainlemma}, we require $n' \leq \frac{n}{2}$ because of the bound $|G / \langle k \rangle | \leq n/2$ for any element $k \in G$ with $k \neq 0$. However, if $n$ is odd, then we know that $|G / \langle k \rangle | \leq n/3$, so we only need to require that $n' \leq \frac{n}{3}$ in this condition. Indeed, if $2$ does not divide $|G|$, then $2$ does not divide $|G / \langle k \rangle |$, and induction may be used. Hence, we may relax the requirement $\frac{c}{2} \geq \frac{d}{\sqrt{2}}$ from Lemma \ref{undirectedcondition2} to $\frac{c}{2} \geq \frac{d}{\sqrt{3}}$.
    
        Then, minimizing $d$ with respect to $\frac{c}{2}\geq \frac{d}{\sqrt{3}}$ and $d \geq \frac{1}{ce}+\frac{c}{2}$ yields the solution $c = \sqrt{\frac{1}{(\sqrt{3}-1)e}} \approx 1.0025$ and $ d = \sqrt{\frac{3}{2(\sqrt{3}-1)e}} \approx 0.8682$. The result then follows as in Theorem \ref{undirectedupperbound}.
        
        \item As $2$ and $3$ do not divide $n$, in condition (2) of Lemma \ref{mainlemma}, we only need to require $n' \leq \frac{n}{5}$. Hence, we may relax the requirement $\frac{c}{2} \geq \frac{d}{\sqrt{2}}$ from Lemma \ref{undirectedcondition2} to $\frac{c}{2} \geq \frac{d}{\sqrt{5}}$. Then, minimizing $d$ with respect to $\frac{c}{2}\geq \frac{d}{\sqrt{5}}$ and $d \geq \frac{1}{ce}+\frac{c}{2}$ yields the solution $c = d =  \sqrt{\frac{2}{e}} \approx 0.8578$. The result then follows as in Theorems \ref{undirectedupperbound}.
    \end{enumerate}
\end{proof}

 We note than no further improvement based on $p$ is possible, as $c = d =  \sqrt{\frac{2}{e}} \approx 0.8578$ is the optimal solution when ignoring the constraint $\frac{c}{2} \geq \frac{d}{\sqrt{p}}$. This solution always respects the constraint $\frac{c}{2} \geq \frac{d}{\sqrt{p}}$ when $p\geq 5$, as in those cases $c = d$ and $\sqrt{p}>2$.


\section{Upper bound for directed Cayley graphs}\label{directedsection}
In this section, we consider the game of cops and robbers on directed abelian Cayley graphs. As we consider directed graphs in this section, whenever we have an abelian group $G$ generated by a set $S$, we no longer require that $S = -S$.

We will show that the approach we have outlined in Lemma \ref{mainlemma} gives us an upper bound of $\sqrt{\frac{2}{\left(\sqrt{2}-1\right) e}} \sqrt{n}+2\approx 1.3328\sqrt{n}+2$ on the cop number of directed abelian Cayley graphs of $n$ vertices. In other words, we will show that Meyniel's conjecture still holds for abelian Cayley digraphs, albeit with a worse coefficient than that of Theorem \ref{undirectedupperbound}. Our general approach in this section will be very similar to that of Section \ref{undirectedsection}. We will define functions $g$ and $h$ that satisfy Lemma \ref{mainlemma} for $(\mathcal G_d,c_d)$ and such that $g(n,s,s)$ is not too large. Note that the functions $g$ and $h$ that we will define in this section are not the same as the functions $g$ and $h$ from the previous section. As this section follows the same approach as Section \ref{undirectedsection}, our presentation will be terser. 

In the following proposition, we use Lemma \ref{mainlemma} with $(\mathcal G_d, c_d)$ to establish a directed version of Theorem \ref{thmFrankl}.
The following proposition appears in \cite{HAMIDOUNE1987289}, but
just like Proposition \ref{weakFrankl}, we include the proposition as an instructive example of how to apply Lemma \ref{mainlemma} with $(\mathcal G_d, c_d)$. Furthermore, we will need a result from the proof of the following proposition to prove the main result of this section.

\begin{prop}\label{hamidouneBound}\cite{HAMIDOUNE1987289}
If $\Gamma$ is a Cayley digraph on an abelian group with generating set $S$ such that $0_G\notin S$, then
$$c(\Gamma)\leq |S|+1.$$
\end{prop}

\begin{proof}
We wish to define functions $g$ and $h$ taking values in $\mathcal D \setminus B$ that satisfy the conditions of Lemma \ref{mainlemma} for $\mathcal G_d$ and $c_d$. We choose
	$$g(n,s,t)=t+1;$$
$$h(n,s,t)=1.
$$
We claim that $g$ and $h$ satisfy the conditions of Lemma \ref{mainlemma} for $\mathcal G_d$ and $c_d$. We immediately notice that $h(n,s,t) > 0$ by definition, and for all $(n,s,t) \in \mathcal D \setminus \mathcal B$, we have $t \geq 1$, and hence $g(n,s,t) \geq 2$.

We show that condition (1) of Lemma \ref{mainlemma} is satisfied. Indeed, let $(G,S,T)\in \mathcal G_d$ with respective sizes $(n,s,t) \in \mathcal D \setminus \mathcal B$. For any $a \in T$, $a$ accounts for $a$ with respect to $S$, because there exists $0_G \in S \cup \{0_G\}$ satisfying $a - 0_G = a$. Therefore, $h$ satisfies condition (1) of Lemma \ref{mainlemma}.

Next, we show that condition (2) of Lemma \ref{mainlemma} is satisfied. If $n' \leq n/2$, $s' \leq s$, and $t' \leq t$ are values such that $(n',s',t') \in \mathcal D \setminus \mathcal B$, then $$g(n,s,t) = t+1  \geq t' + 1 = g(n',s',t').$$

Finally, we show that condition (3) of Lemma \ref{mainlemma} is satisfied. If $1 \leq t'\leq t-h(n,s,t)$, then 
$$g(n,s,t) = t + 1 \geq t' + 2 = g(n,s,t')+1.$$

Hence, as $g$ and $h$ satisfy the conditions of Lemma \ref{mainlemma} for $\mathcal G_d$ and $c_d$, it follows that $c_d(n,s,t) \leq g(n,s,t)$ for all $(n,s,t) \in \mathcal D\setminus \mathcal B$. Therefore, for an abelian group $G$ of $n$ elements generated by a set $S \subseteq G$ of $s$ elements, if $\Gamma = \cay(G,S)$, then one of the following holds: either $(n,s,s) \in \mathcal B$ and $c(\Gamma) \leq c_u(n,s,s) = 1 < s+1$, or $(n,s,s) \in \mathcal D \setminus \mathcal B$ and $c(\Gamma) \leq c_u(n,s,s) \leq g(n,s,s) = s+1$. 
\end{proof}

 	We define a function $h$ that satisfies condition (1) of Lemma \ref{mainlemma} for $\mathcal G_d$ and $c_d$, and we will use this definition of $h$ throughout the entire section. The definition of $h$ contains a constant $c > 0$ whose value we will decide later. 
 
 \begin{defn}
	We define the function $h:\mathcal D\setminus \mathcal B\rightarrow \R^{> 0}$ by
	$$h(n,s,t)=
\begin{cases}
1& t\leq c\sqrt{n}\\
\frac{ts}{n-1}&t>c\sqrt{n}.
\end{cases}$$
\end{defn}

	 \begin{lem}\label{directedfunction}
The function $h$ satisfies condition (1) of Lemma \ref{mainlemma} for $(\mathcal G_d, c_d)$.
\end{lem}
\begin{proof}
	We must show that if $(G,S,T) \in \mathcal C_d$ is a triple with respective sizes $(n,s,t) \in \mathcal D \setminus \mathcal B$, then there exists an element $k \in S\cup\{0_G\}$ accounting for at least $h(n,s,t)$ elements of $T$ with respect to $S$.
When $t \leq c \sqrt{n}$, the proof follows the method of Proposition \ref{hamidouneBound}. When $t > c \sqrt{n}$, the proof follows the method of Lemma \ref{undirectedfunction}.
\end{proof}
 
	 Next, we define our function $g$. Again, we will show that $g$ and $h$ satisfy the conditions of Lemma \ref{mainlemma} and that $g(n,s,s)$ is not too large.

	\begin{defn}
	We define the function $g:\mathcal D\setminus \mathcal B\rightarrow \R^{\geq 2}$ by
	
$$g(n,s,t)=\begin{cases}
t+1   &  t\leq c\sqrt{n}\\
\frac{\log \frac{t}{c\sqrt n}}{\log \frac{n-1}{n-s-1}}+c\sqrt n +2   &  t>c\sqrt{n}.\\
\end{cases}$$
\end{defn}

The following three lemmas are analogues of Lemmas \ref{undirectedbound}, \ref{undirectedcondition2} ,and \ref{undirectedcondition3}.

\begin{lem}\label{directedbound}
	Let $d > 0$. If $d\geq \frac{1}{ce}+c$, then $g(n,s,t)\leq d\sqrt{n}+2$. 
\end{lem}
\begin{proof}
	We consider two cases :
	\begin{enumerate}
		\item If $t\leq c\sqrt{n}$, then
			$$g(n,s,t)= t + 1 \leq 
			c \sqrt{n} + 1 < d \sqrt{n} + 2.$$
		\item If $t>c\sqrt{n}$, then by the proof of Lemma \ref{undirectedbound}, $ \frac{\log \frac{t}{c\sqrt n}}{\log \frac{n-1}{n-s-1}}\leq \frac{1}{ce} \sqrt{n}$ for all real values $n \geq 3$ and $1 \leq t \leq s \leq n - 2$.
		 Thus, we have
		$$g(n,s,t)=\frac{\log \frac{t}{c\sqrt n}}{\log \frac{n-1}{n-s-1}}+c \sqrt{n}+2\leq \left(\frac{1}{ce}+c\right)\sqrt{n}+2 \leq d\sqrt{n}+2.$$
	\end{enumerate}
\end{proof}

\begin{lem}\label{directedcondition2}
	Let $d > 0$. If $c\geq \frac{d}{\sqrt 2}$ and $d\geq \frac{1}{ce}+c$, then
	$g$ respects condition (2) of Lemma \ref{mainlemma}.
\end{lem}
\begin{proof}
	Consider a choice of $d>0$ such that $c \geq \frac{d}{\sqrt 2}$ and $d\geq \frac{1}{ce}+c$. Let $(n,s,t)\in \mathcal D\setminus \mathcal B$ and $(n',s',t')\in \mathcal D \setminus \mathcal B$ such that $n'\leq \frac{n}{2}$, $s'\leq s$ and $t'\leq t$. We consider two cases:
	\begin{enumerate}
		\item If $t\leq c\sqrt n$, then $$g(n,s,t)=t+1\geq t'+1\geq c_u(n',s',t'),$$
		using a result from the proof of Proposition \ref{hamidouneBound}.
		\item If $t>c\sqrt{n}$, then by the previous lemma and our hypotheses on $c$ and $d$,
		\begin{align*}
			 g(n,s,t)=\frac{\log \frac{t}{c\sqrt n}}{\log \frac{n-1}{n-s-1}}+c \sqrt{n} + 2>  c \sqrt{n}+2 \geq  d\frac{\sqrt{n}}{\sqrt{2}}+2 \geq  d\sqrt{n'}+2 \geq g(n',s',t').
		\end{align*}
	\end{enumerate}
\end{proof}

\begin{lem}\label{directedcondition3}
	The functions $g$ and $h$ respect condition (3) of Lemma \ref{mainlemma} for $(\mathcal G_d, c_d)$.
\end{lem} 

\begin{proof}
Let $(n,s,t)\in \mathcal D\setminus \mathcal B$ and $1\leq t'\leq t-h(n,s,t)$.

We consider three cases :
	\begin{enumerate}
		\item If $1 \leq t\leq c\sqrt{n}$, then $h(n,s,t)= 1$, and thus $t\geq t'+1$. Then,
		$$g(n,s,t)=t + 1 \geq t' + 2= g(n,s,t') + 1.$$
		\item If $t> c\sqrt{n}$ and $t'\leq  c\sqrt{n}$, then
			\begin{align*}
				g(n,s,t)=\frac{\log \frac{t}{c\sqrt n}}{\log \frac{n-1}{n-s-1}}+c \sqrt{n} + 2 > c \sqrt{n} + 2 \geq t' + 2 = g(n,s,t') + 1.
			\end{align*}

		\item If $t,t'>c\sqrt{n}$, we know from Lemma \ref{directedfunction} that $t'\leq t-\frac{ts}{n-1}=t\left(\frac{n-s-1}{n-1}\right)$. Thus,
				$$
				g(n,s,t)=\frac{\log \frac{t}{c\sqrt n}}{\log \frac{n-1}{n-s-1}}  + c \sqrt{n}+ 2  \geq \frac{\log \left(\frac{t'}{c\sqrt n}\cdot  \frac{n-1}{n-s-1}\right)}{\log \frac{n-1}{n-s-1}}  + c \sqrt{n} + 2=\frac{\log \left(\frac{t'}{c\sqrt n} \right)}{\log \frac{n-1}{n-s-1}}+1  + c\sqrt{n} + 2=g(n,s,t')+1.$$ 
	\end{enumerate}
	In all cases, $g(n,s,t) \geq g(n,s,t') + 1$, so condition (3) of Lemma \ref{mainlemma} is satisfied for $(\mathcal G_d, c_d)$.
\end{proof}

\begin{thm}\label{directedupperbound}
	The cop number of any directed Cayley graph on an abelian group of $n$ elements is at most $\sqrt{\frac{2}{\left(\sqrt{2}-1\right) e}} \sqrt{n}+2\approx 1.3328\sqrt{n}+2$.
\end{thm}

\begin{proof}
Let $G$ be an abelian group on $n$ vertices generated by set $S \subseteq G$ (satisfying $0_G\notin S$) of $s$ elements.  If $(n,s,s)\in \mathcal B$, then the result follows directly from Lemma \ref{admissiblelemma}. Otherwise, we assume that $(n,s,s)\in \mathcal D\setminus B$.

We first find values $c$ and $d$ satisfying $c\geq \frac{d}{\sqrt 2}$ and $d\geq \frac{1}{ce}+c$, which minimize $d$. A computation of such values $c$ and $d$ yields $c=\sqrt{\frac{1}{e(\sqrt{2} -1)}} \approx 0.9424$ and
$d = \sqrt{\frac{2}{e(\sqrt{2} -1)}}\approx 1.3328$.

With these chosen values of $c$ and $d$, the lemmas of this section show that our choices for $g$ and $h$ satisfy all three conditions of Lemma \ref{mainlemma} for the case $(\mathcal G_d, c_d)$. Hence, by Lemmas \ref{mainlemma} and \ref{directedbound}, $c(\cay(G,S)) \leq c_d(n,s,s) \leq g(n,s,s)\leq d\sqrt{n}+2$.
\end{proof}

Similarly to Proposition \ref{undirectedprime}, we may obtain marginal improvements on the coefficient of $\sqrt{n}$ by considering the group structure of $G$.

\begin{prop}\label{directedprime}
	Let $G$ be an abelian group on $n$ of vertices, and let $S \subseteq G$ be a generating set of $G$ such that $ 0_G\notin S$. Let $p$ be the smallest prime factor of $n$.
	\begin{enumerate}
		
		\item If $p=3$, then $c(G,S)\leq \sqrt{\frac{3}{\left(\sqrt{3}-1\right) e}}\sqrt{n}+2\approx 1.2278\sqrt{n}+2$.
		\item If $p\geq 5$, then $c(G,S)\leq\frac{2}{\sqrt{e}}\sqrt{n} +2\approx 1.2131\sqrt{n}+2$.
	
	\end{enumerate}
\end{prop}
\begin{proof}
    \begin{enumerate}
        \item As in the proof of Theorem \ref{undirectedupperbound}, we only need to require that $n' \leq \frac{n}{3}$ in condition (2) of Lemma \ref{mainlemma}. Hence, we may relax the requirement $c \geq \frac{d}{\sqrt{2}}$ from Lemma \ref{directedcondition2} to $c \geq \frac{d}{\sqrt{3}}$.
    
        Then, minimizing $d$ with respect to $c\geq \frac{d}{\sqrt{3}}$ and $d \geq \frac{1}{ce}+c$ yields the solution $c=\frac{1}{\sqrt{\left(\sqrt{3}-1\right) e}}\approx 0.7089$ and $d=\sqrt{\frac{3}{\left(\sqrt{3}-1\right) e}}\approx 1.2278$. The result then follows as in Theorem \ref{directedupperbound}.
        
        \item As $2$ and $3$ do not divide $n$, in condition (2) of Lemma \ref{mainlemma}, we only need to require $n' \leq \frac{n}{5}$. Hence, we may relax the requirement $c \geq \frac{d}{\sqrt{2}}$ from Lemma \ref{undirectedcondition2} to $c \geq \frac{d}{\sqrt{5}}$. Then, minimizing $d$ with respect to $c\geq \frac{d}{\sqrt{5}}$ and $d \geq \frac{1}{ce}+c$ yields the solution $c=\frac{1}{\sqrt{e}}\approx 0.6065$ and $d=\frac{2}{\sqrt{e}}\approx 1.2131$. The result then follows as in Theorem \ref{directedupperbound}.
    \end{enumerate}
\end{proof}


\section{Constructions with cop number $\Theta( \sqrt{n})$ }\label{lowersection}

In this section, we will give constructions for abelian Cayley graphs and digraphs on $n$ vertices with cop number $\Theta(\sqrt{n})$. If Meyniel's conjecture is true, then for any graph $G$ on $n$ vertices, the greatest possible cop number of $G$ is of the form $\Theta(\sqrt{n})$. Therefore, for an infinite family $\mathcal G$ of graphs, if for each $n \geq 1$, every graph $G \in \mathcal G$ on $n$ vertices has a cop number of the form $\Theta(\sqrt{n})$, then we say that $\mathcal G$ is a \emph{Meyniel extremal family}.

We will construct a Meyniel extremal family using undirected abelian Cayley graphs and a Meyniel extremal family using directed abelian Cayley graphs. These families will show that the upper bounds in Theorems \ref{undirectedupperbound} and \ref{directedupperbound} are best possible, up to a constant factor. Our constructions will be based on finite fields. We note that in \cite{hasiri}, Hasiri and Shinkar use similar methods to construct Meyniel extremal families of undirected abelian Cayley graphs, and the largest cop number of a graph on $n$ vertices by their construction is $\sqrt{\frac{n}{5}}$. Our Meyniel extremal family of undirected abelian Cayley graphs will give a sharper lower bound and thus improve the results from \cite{hasiri}.

In this section, when we consider an abelian group $G$ generated by a set $S$, we will assume that $0_G\in S$, as this will simplify our notation and our arguments. Then, we consider a ``non-move" of a cop or robber to be equivalent to playing the move $0_G$. Hence, we will assume that on a given move, each cop or robber chooses a move $s \in S$ and plays $s$, and we will not give ``non-moves" special treatment.

We will now define the abelian groups and generating sets used to construct our Meyniel extremal families. 
	Let $p>3$ be a prime, and let $G$ be the additive group $(\mathbb{Z}/p \mathbb{Z})^2$. Note that $G$ is in fact a field equipped with a multiplication operation.
	Let $S_1$ and $S_2$ be defined as follows:
	$$S_1 = \{(x,x^3): x \in \mathbb{Z} / p \mathbb{Z}\},$$ 
	$$S_2 = \{(x,x^2): x \in \mathbb{Z} / p \mathbb{Z}\}.$$

We note that our sets $S_1$ and $S_2$ appear as examples of Sidon subsets for certain finite abelian groups in a paper by Babai and S\'{o}s \cite{Babai}. We will see that our proofs that these generating sets give Cayley graphs of high cop number will be similar to the original arguments from \cite{Babai} showing that these sets are Sidon subsets.

It is straightforward to show that $S_1$ and $S_2$ are both generating sets of $G$, seen as a group. We note that $S_1$ is also closed under inverses, while $S_2$ is not closed under inverses in general. Therefore, we consider $\cay(G,S_1)$ to be an undirected abelian Cayley graph, and we consider $\cay(G,S_2)$ to be a directed abelian Cayley graph. We note that $|G| = p^2$. The next two theorems show that both $\cay(G,S_1)$ and $\cay(G,S_2)$ have a cop number of the form $\Theta(p)$, demonstrating that our constructions indeed give graphs and digraphs on $n$ vertices with cop number $\Theta(\sqrt{n})$.

We note that the proofs of the following theorems use key ideas from Proposition 2 and the subsequent discussion of \cite{FranklCayley} and Proposition 2.1 of \cite{HAMIDOUNE1987289}, specifically about the number of moves that a single cop can guard. In particular, we could shorten our proofs and refer directly to those results, but we nonetheless present the full proofs for the sake of completeness.

\begin{thm}
Let $G$, $S_1$, and $p$ be as in the construction above. Then the cop number of $\cay(G,S_1)$ is exactly $\lceil \frac{1}{2} p \rceil=\left \lceil \frac{1}{2} \sqrt{|G|} \right \rceil$.
\label{lowerBound}
\end{thm}

\begin{proof}
We first give a lower bound for the cop number of $\cay(G,S_1)$. Whenever a cop is able to capture the robber immediately after the robber plays a move $(x,x^3)$, we say that the cop guards the move $(x,x^3)$. We show that a single cop cannot simultaneously guard more than two robber moves. Let $v \in G$ be a vertex occupied by a cop $C$, and let $r \in G$ be the vertex occupied by the robber. If the robber is not yet caught, then $v - r = (a,b)$, for some elements $a$ and $b$ that are not both zero. If $C$ guards a move $(x,x^3) \in S_1$, then there must exist a move $(y, y^3) \in S_1$ by which $C$ can capture the robber in reply to $(x,x^3)$. It then follows that $(x,x^3) - (y,y^3) = (a,b)$. Thus $x$ and $y$ must satisfy
$$x - y = a$$
$$x^3 - y^3 = b.$$
By substitution, we obtain the equation 
$$a^3 - 3a^2x + 3ax^2 = b.$$
We see that if $a \neq 0$, then the system of equations has at most two solutions; otherwise, $a = b = 0$. Therefore, for fixed elements $a$ and $b$ not both equal to $0$, there exist at most two values $x$ for which a solution to the system of equations exists. Hence $C$ guards at most two robber moves $(x,x^3) \in S_1$. 

The robber has a total number of moves equal to $|S_1| = p = \sqrt{|G|}$. If the total number of cops is less than $\frac{1}{2}p$, then the robber will always have some move that is not guarded by any cop. Then by naively moving to an unguarded vertex on each turn, the robber can evade capture forever. Hence the cop number of $\cay(G,S_1)$ is at least $\frac{1}{2}p = \frac{1}{2} \sqrt{|G|}$. As cop number is an integer, the cop number of $\cay(G,S_1)$ therefore is at least $\lceil \frac{1}{2}p \rceil$. It follows from Theorem \ref{thmFrankl} that the cop number of $\cay(G,S_1)$ is exactly $\lceil \frac{1}{2}p \rceil$.
\end{proof}

We now show an analoguous result for directed graphs.

\begin{thm}
Let $G$, $S_2$, and $p$ be as in the construction above. Then the cop number of the directed graph $\cay(G,S_2)$ is equal to $|S_2| = p = \sqrt{|G|}$.
\label{directedLowerBound}
\end{thm}

\begin{proof}
We first give a lower bound for the cop number of $\cay(G,S_2)$.
Whenever a cop is able to capture the robber immediately after the robber plays a move $(x,x^2)$, we say that the cop guards the move $(x,x^2)$. We show that a single cop cannot guard more than one robber move. Let $v \in G$ be a vertex occupied by a cop $C$, and let $r \in G$ be the vertex occupied by the robber. If the robber is not yet caught, then $v - r = (a,b)$, for some elements $a$ and $b$ that are not both zero. If $C$ guards a move $(x,x^2)$, then there must exist a move $(y, y^2)$ by which $C$ can capture the robber in reply to $(x,x^2)$. It then follows that $(x,x^2) - (y,y^2) = (a,b)$. Thus $x$ and $y$ must satisfy
$$x - y = a$$
$$x^2 - y^2 = b.$$

By substitution, we obtain the equation $a^2 -2ax = b$, from which we see that whenever $a \neq 0$, $x$ is uniquely determined; otherwise $a = b = 0$. Therefore, for fixed elements $a$ and $b$ not both equal to $0$, there exists exactly one value $x$ for which a solution to the system of equations exists. Hence the cop occupying $C$ guards at most one robber move $(x,x^2) \in S_2$. 

The robber has a total number of moves equal to $|S_2| = p = \sqrt{|G|}$. If the total number of cops is less than $p$, then the robber will always have some move that is not guarded by any cop. Then by naively moving to an unguarded vertex on each turn, the robber can evade capture forever. Hence the cop number of $\cay(G,S_2)$ is at least $|S_2| = p = \sqrt{|G|}$. It follows from Theorem \ref{hamidouneBound} that the cop number of $\cay(G,S_2)$ is exactly $p$.
\end{proof}

Our construction in Theorem \ref{directedLowerBound} implies that if Meyniel's conjecture holds for strongly connected directed graphs, written as $c(\Gamma)\leq c\sqrt{n}$, then the coefficient must respect $c\geq 1$. Although our construction in Theorem \ref{directedLowerBound} uses a digraph whose order is the square of a prime, by using a common argument based on the density of primes (c.f. \cite[Corollary 4.2]{pawelRandom}), we may extend our construction to give a digraph with an order of any large integer $n$ and a cop number of $(1 - o(1)) \sqrt{n}$. This will give us a Meyniel extremal family of digraphs. It is shown in \cite{Baird, BonatoBurgess, Seamone, pawelRandom} that there exist graph and digraph families on $n$ vertices with cop number $\Omega(\sqrt{n})$, but to the authors' knowledge, our multiplicative coefficient of $1-o(1)$ is the largest of any digraph construction.

\begin{coro}
\label{corAllN}
For $n$ sufficiently large, there exist a strongly connected directed graph on $n$ vertices with cop number at least $\sqrt{n - 2n^{0.7625}}  = (1-o(1)) \sqrt{n}$.
\end{coro}
\begin{proof}
We borrow a lemma from number theory which tells us that for $x$ sufficiently large, there exists a prime in the interval $[x-x^{0.525}, x]$ \cite{Baker}. From this lemma it follows that for sufficiently large $x$, there exists a square of a prime in the interval $[x - 2x^{0.7625}, x]$.

For our construction, we let $n$ be sufficiently large, and we choose a prime number $p > 3$ with $p^2 \in [n - 2n^{0.7625}, n]$. We let $G = (\mathbb{Z}/p\mathbb{Z})^2$, and we let $S_2$ be as in Theorem \ref{directedLowerBound}. We then attach a sufficiently long bidirectional path to one of the vertices of $\cay(G,S_2)$, which increases the number of vertices without changing the cop number. This gives us a strongly connected directed graph on $n$ vertices with cop number equal to $c(G,S_2) = p \geq \sqrt{n - 2n^{0.7625}} = (1 - o(1))\sqrt{n}$.
\end{proof}

By using a similar approach, the construction in Theorem \ref{lowerBound} can be modified to give a Meyniel extremal family of undirected graphs on $n$ vertices with cop number $(\frac{1}{2} - o(1)) \sqrt{n}$. However, this lower bound is not best possible, as constructions from \cite{BonatoBurgess} and \cite{pawelRandom} show that there exist undirected graph families in which a graph on $n$ vertices has cop number $(\frac{\sqrt{2}}{2} - o(1)) \sqrt{n}$.

	
\section{Further directions}\label{furtherdirections}
We conjecture that the constructions given in Theorems \ref{lowerBound} and \ref{directedLowerBound} have greatest possible cop number in terms of $n$, up to an additive constant.
\begin{conj}
The cop number of any undirected Cayley graph on an abelian group of $n$ elements is at most $\frac{1}{2} \sqrt{n} + O(1)$.
\end{conj}
\begin{conj}
The cop number of any directed Cayley graph on an abelian group of $n$ elements is at most $\sqrt{n} + O(1)$.
\end{conj}

There are multiple possible avenues of improvement on the proofs of this article. One obvious improvement would be to improve our bounds on the number of robber moves that can be accounted for by one group element. In the explanation behind the choice of $g$, the inequality can be strengthened to be $z_i\geq z_{i-1}+\left\lceil\frac{s(t-z_{i-1})}{n-i}\right\rceil$, as $z_i$ is always an integer and as we can apply the pigeonhole argument only to choose elements which have not previously been chosen. Resolution of this recursion might suggest a better function.

We note that as $g$ and $h$ are defined over integers, the proofs of our upper bounds only depend on the sizes of $G$, $S$ and $T$. Another possible improvement would be to use other group properties of $G$, $S$, and $T$ to get better bounds on the number of robber moves that a single element $k \in G$ can account for, or to better characterize the structure of a quotient $G/\langle k\rangle$ in our inductive strategy.

	
\section*{Acknowledgements}
We thank Ladislav Stacho and Geňa Hahn for bringing this problem to our attention. We also thank Ladislav Stacho, Bojan Mohar, Geňa Hahn, and Ben Seamone for fruitful discussions and encouragement. We furthermore thank Matt Devos for his suggestion to consider Sidon subsets for constructing abelian Cayley graphs with cop number $\Theta(\sqrt{n})$, Frank Ramamonjisoa for his idea of obtaining improved bounds by considering prime factorizations and for his help in organizing earlier versions of the article, as well as Julien Codsi, Alizée Gagnon and Simon St-Amant for helpful comments. Finally, we thank the referees for their comments, which have greatly improved the presentation of this article.


\raggedright
\bibliographystyle{abbrv}
\bibliography{refs}

\end{document}